\documentclass[10pt]{amsart}
\usepackage[utf8]{inputenc}

\usepackage{amssymb,amsthm,amsmath,mathrsfs}
\usepackage{enumerate}
\usepackage{graphicx,color}
\usepackage{setspace}
\usepackage[hidelinks]{hyperref}

\usepackage{tikz,tkz-euclide}
\usetikzlibrary{arrows,calc}
\usepackage{pgfplots}
\pgfplotsset{compat=1.18}
\usetikzlibrary{patterns}% <- added
\usepgfplotslibrary{fillbetween}
\usetikzlibrary{intersections, pgfplots.fillbetween}
\usetikzlibrary{decorations.markings}
\usepackage{bigints}
\usepackage{cleveref}
\usepackage[hypcap=false]{caption}
\usepackage{color}

\newcommand{\dd}{\mathrm{d}}
\newcommand{\E}{\mathbb{E}}
\newcommand{\1}{\textbf{1}}
\newcommand{\R}{\mathbb{R}}
\newcommand{\C}{\mathbb{C}}
\newcommand{\D}{\mathbb{D}}

\newcommand{\p}[1]{\mathbb{P}\left( #1 \right)}
\newcommand{\scal}[2]{\!\left\langle #1, #2 \right\rangle\!}
\newcommand{\red}{}

\DeclareMathOperator{\vol}{vol}

\usepackage{xpatch}
\makeatletter
\def\thm@space@setup{%
  \thm@preskip=12pt plus 0pt minus 0pt
  \thm@postskip=0pt plus 0pt minus 0pt
}
\xpatchcmd{\proof}{6\p@\@plus6\p@\relax}{\z@skip}{}{}
\makeatother

\newtheorem{theorem}{Theorem}
\newtheorem{lemma}[theorem]{Lemma}
\newtheorem{corollary}[theorem]{Corollary}

\theoremstyle{remark}
\newtheorem{remark}{Remark}

\theoremstyle{definition}

\title{\vspace{-3em}
Stability of polydisc slicing
}

\author{Nathaniel Glover}

\author{Tomasz Tkocz}

\author{Katarzyna Wyczesany}
\address{Carnegie Mellon University, Pittsburgh, PA 15213, USA.}
\email{$\{$nsglover, ttkocz, kwyczesa$\}$@andrew.cmu.edu} 

\thanks{Research was supported in part by the NSF grant DMS-2246484.}

\begin{document}

\begin{abstract}
We prove a dimension-free stability result for polydisc slicing due to Oleszkiewicz and Pe\l{}czy\'nski (2000). Intriguingly, compared to the real case, there is an additional asymptotic maximiser. In addition to Fourier-analytic bounds, we crucially rely on a self-improving feature of polydisc slicing, established via probabilistic arguments.
\end{abstract}

\maketitle

\bigskip

\begin{footnotesize}

\end{footnotesize}

\bigskip

\section{Introduction}

The study of sections of convex bodies has a long and rich history. Many results about extremal sections and their stability are known (see the recent survey \cite{NT-surv} and the references therein). An influential result of this type is Ball's  cube slicing theorem from \cite{Ba}, which states that the hyperplane sections of the unit volume cube $[-\tfrac{1}{2}, \tfrac{1}{2}]^n$ in $\R^n$ have volume bounded between $1$ and $\sqrt{2}$ (the lower bound had been known earlier and goes back to the independent works \cite{Hadw} of Hadwiger and \cite{Hen} of Hensley). Ball's upper bound famously gave a simple counter-example to the Busemann-Petty problem in dimensions $n \geq 10$ (see \cite{Ball-BP, BP, Gar-tom, Kol-mon}). For many other ensuing works, see for instance \cite{Amb, Ball2, BFG, BK, KK, KoKo, KRud, MR-maj, MP, NP, Pou1, Pou2, Vaa}, as well as the comprehensive surveys \cite{NT-surv, Zong}. Both bounds for cube slicing are sharp, the lower one uniquely attained at hyperplanes orthogonal to the vectors $e_i$, $1 \leq i \leq n$, the upper bound uniquely attained at hyperplanes orthogonal to the vectors $e_i \pm e_j$, $1 \leq i < j \leq n$, where $e_1, \dots, e_n$ are the standard basis vectors in $\R^n$. However, only recently quantitative stability results have been developed: for every hyperplane $a^\perp$ in $\R^n$ orhogonal to the unit vector $a$ in $\R^n$ with $a_1 \geq a_2 \geq \dots \geq a_n \geq 0$, we have
\begin{equation}\label{eq:stab-real}
1 + \frac{1}{54}|a-e_1|^2 \leq \vol_{n-1}([-\tfrac12, \tfrac12]^n \cap a^\perp) \leq \sqrt{2} - 6\cdot 10^{-5}\left|a-\frac{e_1+e_2}{\sqrt{2}}\right|,
\end{equation}
where here and throughout this paper $|\cdot|$ denotes the standard Euclidean norm on $\R^n$. A~\emph{local} version of the upper bound has been established by Melbourne and Roberto in \cite{MR} (with applications in information theory), whilst the stated lower and upper bounds are from \cite{CNT} (with the numerical value of the constant in the upper bound from \cite{ENT}, where it is instrumental in extending Ball's cube slicing to the $\ell_p$ balls for $p > 10^{15}$). Distributional stability of Ball's inequality has been very recently studied in \cite{ENT3}.

The goal of this paper is to derive a complex analogue of \eqref{eq:stab-real}. Across  the areas in convex geometry, significant efforts have been made to extend many fundamental and classical results well-known from real spaces to complex ones. For example, see \cite{Ar, Ball-Cplank, BC-E, Bern, Brz, C-E1, C-E2, Esk,  KoKo, KRud, KKZ, KPZ, KZ} (sometimes complex-counterparts turn out to be ``easier'', e.g. \cite{JT, NT, Rot}, but for certain problems, on the contrary, satisfactory results have been elusive, e.g. \cite{Tko}).
A counterpart of Ball's cube slicing in $\C^n$ was discovered  by Oleszkiewicz and Pe\l czy\'nski in \cite{OP}. Let $\D$ be the unit disc in the complex plane and let 
\[
\D^n = \D \times\dots\times \D = \{z \in \C^n,\ \max_{j \leq n} |z_j| \leq 1\}
\]
be the polydisc in $\C^n$, the complex analogue of the cube. For $z, w \in \C^n$, we let as usual $\scal{z}{w} = \sum_{j=1}^n z_j\bar w_j$ be their standard inner product. Oleszkiewicz and Pe\l czy\'nski proved that for every (complex) hyperplane $a^\perp = \{z \in \C^n, \ \scal{z}{a} = 0\}$ orthogonal to the vector $a$ in $\C^n$, we have
\begin{equation}\label{eq:OP}
1 \leq \frac{1}{\pi^{n-1}}\vol_{2n-2}(\D^n \cap a^\perp) \leq 2.
\end{equation}
Interestingly, this is in fact formally a \emph{generalisation} of Ball's result (see Szarek's argument in Remark 4.4 in \cite{OP}).
The lower bound is attained uniquely at hyperplanes orthogonal to the standard basis vectors $e_j$, $1 \leq j \leq n$, the upper one is attained uniquely at hyperplanes orthogonal to the vectors $e_j + e^{it}e_k$, $1 \leq j < k \leq n$, $t \in \R$. In this setting, we identify $\C^n$ with $\R^{2n}$ via the standard embedding and $\vol$ is always Lebesgue measure on the appropriate subspace whose dimension is usually indicated in the lower-script (as for instance here $a^\perp$ becomes a subspace in $\R^{2n}$ of real dimension $2n-2$). Note that, in particular, $\vol_{2n-2}(\D^{n-1}) = \pi^{n-1}$ (obtained as the canonical section $\D^n \cap (1,0,\dots,0)^\perp$), which is the normalising factor above. Thanks to the symmetries of $\D^n$ under the permutations of the coordinates as well as complex rotations along axes $z \mapsto (e^{it_1}z_1, \dots, e^{it_n}z_n)$, it suffices to consider real nonnegative vectors with say nonincreasing components.  The main result of this paper is the following dimension-free stability result which refines \eqref{eq:OP}. {\red It is natural to introduce the normalised section function,
\[
A_n(a) = \frac{1}{\pi^{n-1}}\vol_{2n-2}(\D^n \cap a^\perp), \qquad a \in \R^n,
\]
so that $A_n(e_1) = \frac{1}{\pi^{n-1}}\vol_{2n-2}(\D^{n-1}) = 1$.}

\begin{theorem}\label{thm:main}
For $n \geq 2$ and every unit vector $a$ in $\R^n$ with $a_1 \geq a_2 \geq \dots \geq a_n \geq 0$, we have
\red{
\begin{equation}\label{eq:stab-low}
A_n(a) \geq 1 + \frac{1}{8}|a-e_1|^2,
\end{equation}
as well as}
\begin{equation}\label{eq:main}
A_n(a) \leq 2 - \min\left\{10^{-40}\left|a - \frac{e_1+e_2}{\sqrt{2}}\right|, \ \frac{1}{76}\sum_{j=1}^n a_j^4\right\}.
\end{equation}
\end{theorem}

We do not try to optimise the numerical values of the constants involved (for the sake of clarity). Before we move to proof, several remarks are in place.

\begin{remark}\label{rem:asympt-max}
In contrast to the real case, the deficit term in our upper bound \eqref{eq:main} is more complicated and features the minimum over two quantities: the distance to the \emph{unique} extremiser and the $\ell_4$ norm of $a$. The latter appears to account for the fact that 
\[
\lim_{n \to \infty} A_n\left(\tfrac{1}{\sqrt{n}},\dots,\tfrac{1}{\sqrt{n}}\right) = 2.
\]
In other words, curiously, polidysc slicing admits an additional \emph{asymptotic} (Gaussian) extremiser $(\tfrac{1}{\sqrt{n}},\dots,\tfrac{1}{\sqrt{n}})^\perp$, $n \to \infty$. In the real case, 
\[
\lim_{n \to \infty} \vol_{n-1}\left(\left[-\tfrac{1}{2},\tfrac{1}{2}\right]^n \cap \left(\tfrac{1}{\sqrt{n}},\dots,\tfrac{1}{\sqrt{n}}\right)^\perp\right) = \sqrt{\frac{6}{\pi}} < \sqrt{2}.
\]
\end{remark}

\begin{remark}\label{rem:opt1}
Up to the absolute constants, \eqref{eq:main} is sharp, in that the asymptotic behaviour of the right hand side as a function of the quantities involved $|a - \frac{e_1+e_2}{\sqrt{n}}|$ and $\sum_{j=1}^n a_j^4$ is best possible. Indeed, for the former quantity, consider vectors $a = (\sqrt{\frac{1}{2}+\epsilon}, \sqrt{\frac12-\epsilon},0,\dots,0)$ and note that, by combining \eqref{eq:vol-via-xi} and Lemma~\ref{lm:X+Y}, we get $A_n(a) = \left(\frac12+\epsilon\right)^{-1} = 2 - \epsilon + O(\epsilon^2)$ as $\epsilon\to 0$, whilst the left hand side is $2-\Theta(\epsilon)$. For the latter quantity, testing with $a = \left(\frac{1}{\sqrt{n}},\dots,\frac{1}{\sqrt{n}}\right)$ gives the right hand side of the order $2 - \Theta(\frac{1}{n})$, whilst $A_n(a) = \frac{1}{2}\int_0^\infty \left(\frac{2J_1(t/\sqrt{n}}{t/\sqrt{n}}\right)^n t\dd t$ (see Section \ref{sec:int-ineq} below) which, by using the power series expansion (the definition) of the Bessel function, $\frac{2}{t}J_1(t) = 1 - \frac{t^2}{8} + \frac{t^4}{3\cdot 2^6} + O(t^6)$, $t \to 0$, leads to $A_n(a) = 2 - \Theta(\frac{1}{n})$ as well, as $n \to \infty$.
\end{remark}

{\red
\begin{remark}\label{rem:opt2}
The presence of the asymptotic extremiser also attests to the fact that bound \eqref{eq:main} with a better term $\sum_{j=1}^n |a_j|^p$ with some $p < 4$ in place of $\sum_{j=1}^n a_j^4$ would not hold. Indeed, for $a = \left(\frac{1}{\sqrt{n}},\dots,\frac{1}{\sqrt{n}}\right)$, $A_n(a) = 2 - \Theta(\frac{1}{n})$, as explained in Remark \ref{rem:opt1}, whereas $\sum_{j=1}^n |a_j|^p = n^{1-p/2} \gg n^{-1}$, if $p < 4$.
\end{remark}
}

\section{A sketch of our approach}

{\red The lower bound is established by quantifying a simple convexity argument leading to the main term (akin to the real case, as done in \cite{CNT}).}

For the upper bound, we principally follow the strategy developed in \cite{CNT} (see also Section 5 in \cite{ENT}). 
However, the presence of the asymptotic extremiser (see Remark \ref{rem:asympt-max}) is a new obstacle.
To wit, there are several entirely different arguments, depending on the hyperplane $a^\perp$ (in what follows we always assume as in Theorem \ref{thm:main} that $a$ is a unit vector with nonnegative nonincreasing components). Here is a rough roadmap.

\begin{enumerate}[(a)]
\item When $a$ is \emph{close} to the extremiser $\frac{e_1+e_2}{\sqrt{2}}$, we reapply polydisc slicing in a lower dimension to a portion of $a$, which yields its self-improvement and gives a quantitative deficit (this is largely inspired by a similar phenomenon for Szarek's inequality from \cite{Sza} discovered in \cite{DDS}). This part crucially uses probabilistic insights put forward in \cite{CKT, CNT, CST} and perhaps constitutes the most subtle point of the whole analysis.

\item When $a$ has all coordinates \emph{well} below $\frac{1}{\sqrt{2}}$, we employ Fourier-analytic bounds and quantitative versions of the Oleszkiewicz-Pe\l czy\'nski integral inequality for the Bessel function. This results in the $\ell_4$ norm quantifying the improvement \emph{near} the asymptotic extremiser.

\item When $a$ has one coordiate around $\frac{1}{\sqrt{2}}$ and the others small, $a$ is neither close the the extremiser $\frac{e_1+e_2}{\sqrt{2}}$, nor the Fourier-analytic bounds are applicable. We rely on probabilistic insights again and use a Berry-Esseen type bound.

\item 
When $a$ has a coordinate \emph{barely} above $\frac{1}{\sqrt{2}}$, we use a Lipschitz property of the normalised section function and reduce the analysis to the previous cases.

\item
When $a$ has a coordinate \emph{well}-above $\frac{1}{\sqrt{2}}$, we use a projection argument.

\end{enumerate}

\section{Ancillary results and tools}

Since in the proof we consider several cases that require different approaches and tools, this section which includes auxiliary results is split into several subsections.

\subsection{The role of independence}

Our approach, to a large extent, relies on the following probabilistic formula  for the volume of sections of the polydisc, obtained in \cite{Brz} by Fourier-analytic means (see also \cite{CST} for a \emph{direct} derivation): for every $n \geq 1$ and every \emph{unit} vector $a$ in $\R^n$, we have 
\begin{equation}\label{eq:vol-via-xi}
A_n(a) = \E\left|\sum_{k=1}^n a_k\xi_k\right|^{-2},
\end{equation}
where $\xi_1, \xi_2, \dots$ are independent random vectors uniform on the unit sphere $S^3$ in $\R^4$. 

To leverage independence and rotational symmetry in \eqref{eq:vol-via-xi}, we note the following general observation. 

\begin{lemma}\label{lm:X+Y}
Let $d \geq 3$ and let $X$ and $Y$ be independent rotationally invariant random vectors in $\R^d$. Then
\[
\E|X+Y|^{2-d} = \E\min\{|X|^{2-d},|Y|^{2-d}\}.
\]
In particular, in $\R^4$,
\begin{equation}\label{eq:X+Y}
\E|X+Y|^{-2} = \E\min\{|X|^{-2},|Y|^{-2}\}.
\end{equation}
\end{lemma}

The special case of $d=3$ appeared as Lemma 6.6 in \cite{CNT}, whereas for the general case we follow the argument from Remark 15 of \cite{CST} (see also Corollary 17 therein).

\begin{proof}[Proof of Lemma \ref{lm:X+Y}]
Let $\xi_1, \xi_2$ be independent random vectors uniform on the unit sphere $S^{d-1}$ in $\R^d$. By rotational invariance, $X$ and $Y$ have the same distributions as $|X|\xi_1$ and $|Y|\xi_2$. Conditioning on the values of the magnitudes $|X|$ and $|Y|$, it thus suffices to show that for every $a_1, a_2 \geq 0$, we have
\[
\E|a_1\xi_1 + a_2\xi_2|^{2-d} = \min\{a_1^{2-d}, a_2^{2-d}\}.
\]
By homogeneity and symmetry, this will follow from the special case of $a_1 = 1$, $a_2 = t \in (0,1)$. By rotational invariance, we have
\[
h(t) = \E|\xi_1+t\xi_2|^{2-d} = \E |e_1 + t\xi_2|^{2-d} = \frac{1}{\vol_{d-1}(S^{d-1})}\int_{S^{d-1}} |e_1+t\xi|^{2-d} \dd \xi,
\]
(in the sense of the usual Lebesgue surface integral) and our goal is to argue that this equals $1$ for all $0 < t < 1$. Let $F(x) = |x|^{2-d}$. On the sphere, for every $x \in S^{d-1}$, $x$ is the outer-normal, hence the divergence theorem yields
\begin{align*}
\frac{\dd}{\dd t}\int_{S^{d-1}} F(e_1+t\xi) \dd \xi &= \int_{S^{d-1}} \scal{\nabla F(e_1+t\xi)}{\xi} \dd \xi \\
&= \int_{B_2^d} \text{div}_x(\nabla F(e_1+tx)) \dd x \\
&= t\int_{B_2^d} (\Delta F)(e_1+tx) \dd x = 0
\end{align*}
since $\Delta F = 0$ ($e_1+tx$ never vanishes for $x \in B_2^d$, $0 < t < 1$). Noting that clearly $h(0) = 1$, this finishes the proof.
\end{proof}

\subsection{Integral inequality}\label{sec:int-ineq}

Another key ingredient is the Fourier-analytic expression for the section function,
\begin{equation}\label{eq:Ap-Fourier}
A_n(a) = \frac{1}{2}\int_0^\infty \left(\prod_{j=1}^n \frac{2J_1(a_jt)}{a_jt}\right) t\dd t
\end{equation}
(see (5) in \cite{OP}) and, crucially, the resulting upper-bound obtained from H\"older's inequality with $L_{a_j^{-2}}$ norms (this idea perhaps goes back to Haagerup's work \cite{Haa}): for every $n \geq 1$ and every \emph{unit} vector $a$ in $\R^n$, we have
\begin{equation}\label{eq:Fourier-upbd}
A_n(a) \leq 2\prod_{k=1}^n \Psi(a_k^{-2})^{a_k^2},
\end{equation}
where for $s > 0$,
\begin{equation}\label{eq:def-Psi}
\Psi(s) = \frac{s}{4}\int_0^\infty \left|\frac{2J_1(t)}{t}\right|^st \dd t.
\end{equation}
Here $J_1(t) = \frac{t}{2}\sum_{k=0}^\infty \frac{(-1)^k}{2^{2k}k!(k+1)!}t^{2k}$ is the Bessel function (of the first kind) of order $1$. Since $J_1(t) = O(t^{-1/2})$ as $t \to \infty$ (see, e.g. 9.2.1. in \cite{AS}), $\Psi(s)$ is finite for all $s > \frac{4}{3}$ (for $s \leq \frac{4}{3}$, we let $\Psi(s) = \infty$, so that \eqref{eq:Fourier-upbd} formally holds).

Oleszkiewicz and Pe\l czy\'nski's approach crucially relies on the fact that
\[
\sup_{s \geq 2} \Psi(s) = 1,
\]
and that the supremum is attained at $s = 2$ as well as when $s \to \infty$. Implicit in their proof of this subtle claim is the following quantitative version, crucial for us.

\begin{lemma}\label{lm:integral-ineq}
For the special function $\Psi$ defined in \eqref{eq:def-Psi}, we have
\begin{equation}\label{eq:Psi-quant}
\Psi(s) \leq \begin{cases} 1 - \frac{1}{12}(s-2)^2, & 2 \leq s \leq \frac{8}{3}, \\ 1 - \frac{1}{151s}, & s > \frac{8}{3}. \end{cases}
\end{equation}
\end{lemma}
\begin{proof}
When $2 \leq s \leq \frac83$, we have
\[
\Psi(s) \leq \frac{s}{2}e^{-\frac{s-2}{2}},
\]
as showed in \cite{OP} (\emph{Proof of Proposition 1.1 in Case (II)}, p. 290). It remains to apply an elementary bound to $v = \frac{s}{2} - 1 \in [0,\frac13]$,
\[
(v+1)e^{-v} \leq (v+1)(1-v+\tfrac{v^2}{2}) = 1 - \tfrac{v^2}{2} + \tfrac{v^3}{2} \leq 1 - \tfrac{v^2}{3}.
\]

When $s \geq \frac83$, it is showed in \cite{OP} (\emph{Proof of Proposition 1.1 in Case (I)}, p. 288) that
\begin{align*}
\Psi(s) &\leq 1 - \frac{1}{3s} + \frac{1}{3s^2} + \frac{8s}{3s-4}(60\pi^2)^{-s/4} \\
&=1 - \frac{1}{s}\left(\frac{1}{3}-\frac{1}{3s} - \frac{8s^2}{3s-4}(60\pi^2)^{-s/4}\right).
\end{align*}
It remains to note that the function in the bracket is increasing in $s$ on $[\frac83, \infty)$, thus it is at least its value at $s = \frac83$, which is greater than $\frac{1}{151}$.
\end{proof}

%\subsection{Complex intersection bodies}
\subsection{Lipschitz property of the section function and complex intersection bodies}
In perfect analogy to the real case, there is a complex analogue of the classical Busemann's theorem from \cite{Bu} saying that $x \mapsto \frac{|x|}{\vol_{n-1}(K \cap x^\perp)}$ defines a norm on $\R^n$, if $K$ is a symmetric convex body in $\R^n$.

\begin{theorem}[Koldobsky-Paouris-Zymonopoulou, \cite{KPZ}]\label{thm:complex-Bus}
Let $K$ be a complex symmetric convex body $K$ in $\C^n$, that is $K$ is a convex body in $\R^{2n}$ with $e^{it}z \in K$, whenever $z \in K$, $t \in \R$. Then the function
\[
z \mapsto \frac{|z|}{(\vol_{2n-2}(K \cap z^\perp))^{1/2}}
\]
defines a norm on $\C^n$.
\end{theorem}

We use this result to establish a Lipschitz property of the section function $A_n$.

\begin{lemma}\label{lm:Lip}
For unit vectors $a, b$ in $\R^n$, we have
\[
|A_n(a) - A_n(b)| \leq 4\sqrt{2}|a-b|.
\]
\end{lemma}
\begin{proof}
Let $K = (\frac{1}{\pi}\mathbb{D})^n$ be the volume $1$ polydisc, so that $A_n(a) = \vol_{2n-2}(K \cap a^\perp)$. Then, by Theorem \ref{thm:complex-Bus}, $N(a) = |a|A_n(a)^{-1/2}$ is a norm, thus for \emph{unit} vectors $a$ and $b$, we have
\begin{align*}
|A_n(a) - A_n(b)|  = |N(a)^{-2}-N(b)^{-2}| &= \frac{N(a)+N(b)}{N(a)^2N(b)^2}|N(a)-N(b)|\\ &\leq \frac{N(a)+N(b)}{N(a)^2N(b)^2}N(a-b).
\end{align*}
By the definition of $N$, the right hand side becomes
\[
A_n(a)A_n(b)\frac{A_n(a)^{-1/2}+A_n(b)^{-1/2}}{A_n(a-b)^{1/2}}|a-b|
\]
and using the polydisc slicing inequalities, that is $1 \leq A_n(x) \leq 2$ for every vector $x$, the result follows.
\end{proof}

\subsection{Berry-Esseen bound}
Finally, we will employ a Berry-Esseen type bound with explicit constant for random vectors in $\R^4$. Recently, Rai\v c has obtained such a result for an arbitrary dimension.

\begin{theorem}[Rai\v c, \cite{Ra}]\label{thm:Raic}
Let $X_1, \dots, X_n$ be independent mean $0$ random vectors in $\R^d$ such that $\sum_{j=1}^n X_j$ has the identity covariance matrix. Let $G$ be a standard Gaussian random vector in $\R^d$. Then
\[
\sup_A\left|\p{\sum_{j=1}^n X_j \in A} - \p{G \in A}\right| \leq (42d^{1/4}+16)\sum_{j=1}^n \E|X_j|^3,
\]
where the supremum is over all Borel convex sets in $\R^d$.
\end{theorem}

\section{Proof of Theorem \ref{thm:main}}

In this section we will present the proof of the Theorem \ref{thm:main}. We recall that $a$ is assumed to be a unit vector in $\R^n$ such that $a_1\ge a_2\ge \ldots \ge a_n \geq 0$. 

{\red 
We begin with a short proof of the lower bound \eqref{eq:stab-low}. First note that using \eqref{eq:vol-via-xi} and the convexity of $(\cdot)^{-1}$ (Jensen's inequality),
\[
A_n(a) = \E\left(\left|\sum_{k=1}^n a_k\xi_k\right|^2\right)^{-1} \geq \left(\E\left|\sum_{k=1}^n a_k\xi_k\right|^2\right)^{-1} = \left(\sum_{k=1}^n a_k^2\right)^{-1} = 1,
\]
which gives the sharp lower bound without the error term. This of course can be easily improved upon (the same idea is used in the proof of Theorem 6.1 in \cite{CNT}). We let $Y = 2\sum_{k < l} a_ka_l\scal{\xi_k}{\xi_l}$ so that 
\[
\left|\sum_{k=1}^n a_k\xi_k\right|^2 = 1 + Y.
\]
We have an elementary inequality $(1+y)^{-1} \geq 1-y+\frac34y^2-\frac14y^3$, $y > -1$ (after simplifications, equivalent to  $\frac14y^2(y-1)^2 \geq 0$). This leads to the bound
\[
A_n(a) \geq 1 - \E Y + \tfrac34\E Y^2 - \tfrac14\E Y^3. 
\]
Plainly, $\E Y = 0$ (by symmetry). Moreover, it was shown in the course of the proof of Theorem 6.1 in \cite{CNT} that $\E Y^3 \leq \E Y^2$ (the case $d=4$ therein) and $\E Y^2 \geq \frac14|a-e_1|^2$. These result in \eqref{eq:stab-low}.
}

{\red We move on to the upper bound \eqref{eq:main}}. Its proof requires considering multiple cases dependent on the size of the two larges coordinates of the vector $a$. 

For the convenience of the reader we include the following pictorial guide to the proof.

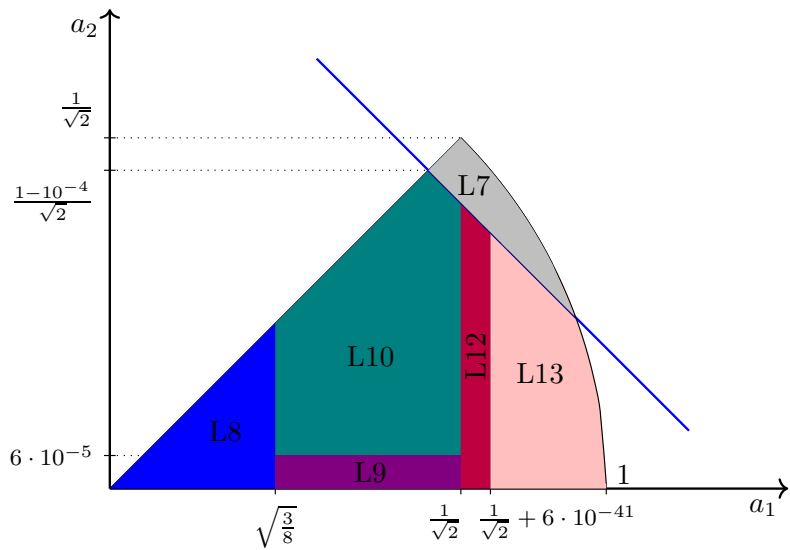
\begin{center}
\scalebox{1.1}{
\begin{tikzpicture}\label{fig:fig1}
% draw x , y lines
\draw[thick,->] (0,0) -- (8.2,0) node[below left] {$a_1$};
\draw[thick,->] (0,0) -- (0,5.8) node[below left] {$a_2$};

% draw x points (Values)
   \draw (2cm,2pt) -- ++ (0,-4pt) coordinate(A)node[below, font=\footnotesize] {$\sqrt{\tfrac{3}{8}}$};
    \draw (4.2426 cm,2pt) -- ++ (0,-4pt) coordinate(B)node[below, xshift=-0.2cm, font=\footnotesize] {$\tfrac{1}{\sqrt{2}}$};
       \draw (4.6cm,2pt) -- ++ (0,-4pt) coordinate(C)node[below, xshift=0.8cm, font=\footnotesize] { $\tfrac{1}{\sqrt{2}} +6\cdot 10^{-41}$};
       \draw (6cm,2pt) -- ++ (0,-4pt) coordinate(D)node[above right] { $1$};
 % draw y points (Values)
  \draw (2pt,0.4cm) -- ++ (-4pt,0) coordinate(E)node[left, font=\footnotesize]  {$6\cdot 10^{-5}$};
  \draw (2pt,4.2426 cm) -- ++ (-4pt,0) coordinate(F)node[above left, font=\footnotesize]  { $\tfrac{1}{\sqrt{2}}$};
  \draw (2pt,3.85cm) -- ++ (-4pt,0) coordinate(G)node[below left, font=\footnotesize]  {$\tfrac{1-10^{-4}}{\sqrt{2}}$};

\draw[dotted] (0,4.2426 cm) -- (4.2426 cm,4.2426 cm);
\draw[dotted] (0,3.85cm) -- (4.2cm,3.85cm);
\draw[dotted] (0,0.4cm) -- (2cm,0.4cm);

  %lines
   \draw[ domain=0:(3*2^0.5), smooth, variable=\x] plot ({\x}, {\x});
   \draw[ domain=2:3*2^0.5, smooth, variable=\x] plot ({\x}, {0.4});
    \draw[ domain=0:2, smooth, variable=\x] plot ({2},{\x});
   \draw[ domain=0:3*2^0.5, smooth, variable=\x] plot ({3*2^0.5},{\x});
    \draw[ domain=0:3.85, smooth, variable=\x] plot ({4.6},{\x});

%fill
 \fill[blue,opacity=0.3](0,0)--(2,0)--(2,2)--cycle;
     \fill[ color=teal, opacity=0.5](2,0.4)--(4.2426,0.4)--(4.2426,4.2426)--(2,2)--cycle;
     \fill[purple, opacity=0.5](4.2426,4.2426)--(4.2426,0)--(4.6,0)--(4.6,3.85)--cycle;
% \fill[teal](4.6,0)--(6,0)--(4.6,3.85)--cycle;
      \fill[violet,opacity=0.5](2,0)--(4.2426,0)--(4.2426,0.4)--(2,0.4)--cycle;

      \fill [pink, opacity=0.5, domain=4.6:6, variable=\x]
      (4.6,0)
      -- plot ({\x}, {(36-\x*\x)^0.5})
      -- (6,0)
      -- cycle;
  \draw[name path=cap,thick, domain=2.5:7, blue] plot ({\x},{-\x+7.7});
  \draw[name path=circ2, smooth, black, domain=4.2426:6] plot ({\x},{(36-\x*\x)^0.5});

\draw[name path=circ, smooth, black, domain=4.2426:5.35] plot ({\x},{(36-\x*\x)^0.5});
  \draw[name path=line, domain=3.85:5.62] plot ({\x},{-\x+7.7});
\tikzfillbetween[of=circ and line, on layer=main]{lightgray, opacity=1}

%nodes
\node[black] at  (4.4,3.67) {L7} ;
\node at  (1.4,0.7) {L8} ;
\node at  (3.15,0.2) {L9} ;
\node at  (3.15,1.6) {L10} ;
\node[black] at  (5.2,1.4) {L13} ;
\node[rotate=90] at (4.42,1.6) {L12} ;
\end{tikzpicture}}
\captionof{figure}{We consider six cases. The labels {\rm L}$k$ correspond to the lemmas in which a given case is resolved.
In Section \ref{sub:self-impr} we explain the case where two largest coordinates are near $\tfrac{1}{\sqrt{2}}$, corresponding to {\rm L7} in the picture above. In Section \ref{sub:all_small} we explain the bound when all cooridinates are below $\sqrt{3/8}$, i.e. we cover the region {\rm L8}. In Section \ref{sub:below1root2} we study the case where $a_1$ is below $1/\sqrt{2}$, which we examine in two regimes depending on the value of $a_2$ corresponding to {\rm L9} and {\rm L10}. We address the case when $a_1$ is only slightly above $\frac{1}{\sqrt{2}}$, marked as {L12}, in Section \ref{sub:a1_just_above}. Finally, in Section \ref{sub:a1_large} we complete the picture by settling the case when $a_1$ is large ({\rm L13}). We put these bounds together, proving the theorem, in Section \ref{sub:proof_main}.
}
\end{center}

\subsection{Two largest coordinates are close to \texorpdfstring{$\tfrac{1}{\sqrt{2}}$}{1/sqrt(2)}: local stability via self-improvement }\label{sub:self-impr}

We set
\begin{align*}
\delta(a) &= \left|a - \frac{e_1+e_2}{\sqrt{2}}\right|^2 = \left(a_1 - \frac{1}{\sqrt{2}}\right)^2 + \left(a_2-\frac{1}{\sqrt{2}}\right)^2 + a_3^2 + \dots + a_n^2 \\
&= 2-\sqrt{2}(a_1+a_2).
\end{align*}
When $n=2$, from Lemma \ref{lm:X+Y}, we have
\[
A_2(a) = \min\{a_1^{-2}, a_2^{-2}\} = a_1^{-2}
\]
and we check that this is at most $2-\sqrt{\delta(a)}$. 
{\red
One way to verify that $a_1^{-2} \leq 2 - \sqrt{\delta(a)}$ is to set $a_1 = \cos\theta$, $\theta \in [0,\frac{\pi}{4}]$, so that then $\delta(a) = 2-\sqrt{2}(\cos\theta + \sin\theta) = 2-2\cos(\frac{\pi}{4}-\theta) = 4\sin^2(\frac{\pi}{8}-\frac{\theta}{2})$. Letting $t = \frac{\pi}{8}-\frac{\theta}{2} \in [0,\frac{\pi}{8}]$, we have $a_1 = \cos(\frac{\pi}{4}-2t)$ and the desired inequality becomes $\frac{1}{\cos^2(\frac{\pi}{4}-2t)} \leq 2(1-\sin t)$. Moreover, $\cos^2(\frac{\pi}{4}-2t) = \frac{1}{2}(\cos(2t)+\sin(2t))^2 = \frac{1}{2}(1+\sin(4t))$, so it suffices to check that $(1-\sin t)(1+\sin(4t)) \geq 1$, $t \in [0,\frac{\pi}{8}]$. Note that on this interval, $\sin(4t) = 4\sin t \cos t\cos (2t) \geq 4\sin t \cos^2 (2t) \geq 2\sin t$ and $(1-\sin t)(1+2\sin t) = 1 +(1-2\sin t)\sin t \geq 1$, since $\sin t < \sin(\frac{\pi}{6}) = \frac{1}{2}$.
}
Hence, Theorem \ref{thm:main} holds when $n=2$. We can assume from now on that $n \geq 3$.

%\subsection*{Local stability via self-improvement}

Our goal here is to establish Theorem \ref{thm:main} for vectors $a$ which are \emph{near} the extremiser. This relies on a self-improving feature of the polydisc slicing result.

\begin{lemma}\label{lm:loc-stab}
We have, $A_n(a) \leq 2 - \frac{1}{25}\sqrt{\delta(a)}$, provided that $\delta(a) \leq \frac{1}{5000}$.
\end{lemma}

\begin{proof}
We let $X = a_1\xi_1  + a_2\xi_2$ and $Y = \sum_{j=3}^n a_j\xi_j$. Then, using \eqref{eq:vol-via-xi}, \eqref{eq:X+Y} and the concavity of $t \mapsto \min\{\alpha, t\}$, we obtain
\[
A_n(a) = \E\min\{|X|^{-2},|Y|^{-2}\} \leq \E_X\min\{|X|^{-2},\E_Y|Y|^{-2}\}.
\]
By polydisc slicing, $\E_Y|Y|^{-2} \leq \frac{2}{1-a_1^2-a_2^2}$. We thus get
\[
A_n(a) \leq \E\min\left\{|X|^{-2}, \frac{2}{1-a_1^2-a_2^2}\right\} = \E|X|^{-2} - \E\left(|X|^{-2}- \frac{2}{1-a_1^2-a_2^2}\right)_+.
\]
Using \eqref{eq:X+Y} again, we get that  $\E|X|^{-2} = \min\{a_1^{-2},a_2^{-2}\} = a_1^{-2}$. 

It will be more convenient to work with the rotated variables
\[
u_1 = \frac{a_1+a_2}{\sqrt{2}}, \qquad u_2 = \frac{a_1-a_2}{\sqrt{2}},
\]
for which $u_1 = 1-\frac{\delta(a)}{2} \in [1-10^{-4}, 1]$, $u_2 > 0$ and $u_1^2+u_2^2 = a_1^2 + a_2^2 < 1$. Then, in terms of $u_1, u_2$, we have
\[\tfrac12 A_n(a) \leq \frac{1}{(u_1+u_2)^2} - \E\left(\tfrac12 |X|^{-2} - \frac{1}{1-u_1^2-u_2^2}\right)_+.
\]

Note also that
\[
|X|^2 = a_1^2+a_2^2 + 2a_1a_2\theta = u_1^2+u_2^2+(u_1^2-u_2^2)\theta,
\]
where $\theta$ is a random variable with density $\frac{2}{\pi}(1-x^2)^{1/2}$ on $[-1,1]$ (the distribution of $\scal{\xi_1}{\xi_2}$ which is the same as the one of $\scal{\xi_1}{e_1}$). We will use this representation in what follows.

Consider two cases:

\emph{Case 1: $u_1^2+9u_2^2 \geq 1$.} We simply neglect the second term (the expectation), to obtain the upper bound of the form
\[
\tfrac12 A_n(a) \leq \frac{1}{(u_1+u_2)^2} \leq \frac{1}{\left(u_1+\sqrt{\frac{1-u_1^2}{9}}\right)^2}.
\]
Denoting for brevity $\delta = \delta(a) \in [0,\frac{1}{5000}]$ we crudely lower-bound the denominator of the right-hand side,
\[u_1+\sqrt{\frac{1-u_1^2}{9}} = 1 - \frac{\delta}{2} + \sqrt{\frac{\delta}{18}\left(2-\frac{\delta}{2}\right)} \geq 1 - \frac{\delta}{2} + \sqrt{\frac{\delta}{10}} \geq 1 + \frac{1}{2}\sqrt{\frac{\delta}{10}}.
\]
Therefore,
\[
A_n(a) \leq 2\left(1 + \frac{1}{2}\sqrt{\frac{\delta}{10}}\right)^{-2} \leq 2\left(1 - \frac{1}{2}\sqrt{\frac{\delta}{10}}\right) = 2 - \sqrt{\frac{\delta(a)}{10}},
\]
where we used that $(1+x)^{-2} \leq 1-x$ holds for $x \in [0,\frac12]$.

\emph{Case 2: $u_1^2 + 9u_2^2 \leq 1$.} 
We use a more refined lower-bound on the expectation, namely
\begin{align*}
\E\left(\tfrac12 |X|^{-2} - \frac{1}{1-u_1^2-u_2^2}\right)_+ &\geq \E\left[\big(\tfrac12 |X|^{-2} - \frac{1}{1-u_1^2-u_2^2}\big)\1_{\left\{\frac12|X|^{-2} \geq \frac{2}{1-u_1^2-u_2^2}  \right\}} \right] \\ &\geq \frac{1}{1-u_1^2-u_2^2} \E\left[\1_{\left\{ {\red \frac12}|X|^{-2} \geq \frac{2}{1-u_1^2-u_2^2} \right\}} \right] \\ &=  \frac{1}{1-u_1^2-u_2^2}\p{|X|^2 \leq \frac{1-u_1^2-u_2^2}{4}}.
\end{align*}

Recalling that $|X|^2 = u_1^2+u_2^2 + (u_1^2-u_2^2)\theta$, the condition $|X|^2 \leq \frac{1-u_1^2-u_2^2}{4}$ becomes $\theta \leq \frac{1-5(u_1^2+u_2^2)}{4(u_1^2-u_2^2)} = -1+\theta_0$ with $\theta_0 = \frac{1-u_1^2-9u_2^2}{4(u_1^2-u_2^2)}$. Note that by our assumption $0<\theta_0$ and that $\theta_0<1$. Indeed, since $u_1>u_2$ and $5(u_1^2+u_2^2) \geq 5u_1^2 = 5(1-\delta/2)^2 \geq 5(1-10^{-4})^2 > 1$ we get that $-1+\theta_0<0$ and the claim follows.
 
Therefore, using that $\theta_0<1$ we estimate the probability of the event $|X|^2 \leq \frac{1-u_1^2-u_2^2}{4}$ by 
\begin{align*}
\p{\theta \leq -1 + \theta_0} = \frac{2}{\pi}\int_{-1}^{-1+\theta_0} \sqrt{1-x^2} \dd x &= \frac{2}{\pi}\int_0^{\theta_0}\sqrt{x(2-x)} \dd x \\
&\geq \frac{2}{\pi}\int_0^{\theta_0}\sqrt{x} \dd x = \frac{4}{3\pi}\theta_0^{3/2}.
\end{align*}
Putting this together and using the fact that $1-u_1^2-u_2^2 \leq 1-u_1^2$ and $u_1^2-u_2^2 \leq 1$, we get
\begin{align}\notag
\tfrac12 A_n(a) &\leq  \frac{1}{(u_1+u_2)^2} - \frac{1}{1-u_1^2-u_2^2}\p{|X|^2 \leq \frac{1-u_1^2-u_2^2}{4}}\\ \notag
&\leq \frac{1}{(u_1+u_2)^2} - \frac{1}{1-u_1^2-u_2^2}\cdot \frac{4}{3\pi}\left(\frac{1-u_1^2-9u_2^2}{4(u_1^2-u_2^2)}\right)^{3/2} \\ \label{eq:u1u2}
&\leq \frac{1}{(u_1+u_2)^2} - \frac{1}{6\pi}\frac{(1-u_1^2-9u_2^2)^{3/2}}{1-u_1^2}.
\end{align}
We claim that the right hand side as a function of $u_2$ is decreasing. Indeed, its derivative equals
\[
-2(u_1+u_2)^{-3} + \frac{9}{2\pi}\frac{u_2(1-u_1^2-9u_2^2)^{1/2}}{1-u_1^2} \leq -2(u_1+u_2)^{-3} + \frac{9}{2\pi}\frac{u_2}{\sqrt{1-u_1^2}}.
\]
Since $1-u_1^2 \geq 9u_2^2$, the second term is at most $\frac{3}{2\pi} < \frac{1}{2}$. Crudely, $u_1 + u_2 = a_1\sqrt{2} < \sqrt{2}$, so the first term is at most $-2\sqrt{2}^{-3} = -\frac{1}{\sqrt{2}}$ and hence the derivative is negative. Setting $u_2 = 0$ in \eqref{eq:u1u2} thus gives
\begin{align*}
\tfrac12 A_n(a) \leq \frac{1}{u_1^2} - \frac{1}{6\pi}\sqrt{1-u_1^2} &= \left(1 - \frac{\delta}{2}\right)^{-2} - \frac{1}{6\pi}\sqrt{\frac{\delta}{2}\left(2-\frac{\delta}{2}\right)} \\
&\leq 1+2\delta - \frac{1}{6\pi}\sqrt{1 - \frac{1}{2}\cdot 10^{-4}}\sqrt{\delta},
\end{align*}
where we have used $(1-x/2)^{-2} \leq 1+ 2x$, $0 \leq x \leq \frac{1}{2}$. Since $\delta \leq \sqrt{\frac{1}{5000}}\sqrt{\delta}$, the right hand side is at most
\[
1+ \left(\frac{2}{\sqrt{5000}}-\frac{1}{6\pi}\sqrt{1 - \frac{1}{2}\cdot 10^{-4}}\right)\sqrt{\delta} < 1 - \frac{\sqrt{\delta}}{50}.\qedhere
\]
\end{proof}

We note for future reference that the complementary case to the one considered in Lemma \ref{lm:loc-stab} is
\begin{equation}\label{eq:assump-delta}
\delta(a) \geq \frac{1}{5000}.
\end{equation}
Since $a_2 \leq \frac{a_1+a_2}{2} = \frac{1-\delta(a)/2}{\sqrt{2}}$, this in particular implies that $a_2$ is bounded away from $\frac{1}{\sqrt{2}}$,
\begin{equation}\label{eq:assumpt-a2}
a_2 \leq \frac{1-10^{-4}}{\sqrt{2}}.
\end{equation}

\subsection{All weights are small}\label{sub:all_small}

When all weights are small and bounded away from $\frac{1}{\sqrt{2}}$, we can rely on the Fourier analytic bound \eqref{eq:Fourier-upbd} because Lemma \ref{lm:integral-ineq} guarantees savings across all weights. This case results with the term $\|a\|_4^4$ in \eqref{eq:main} which quantifies the distance to the \emph{asymptotic extremiser} $a = (\frac{1}{\sqrt{n}}, \dots, \frac{1}{\sqrt{n}})$, $n \to \infty$.

\begin{lemma}\label{lm:all-weights-small}
We have, $A_n(a) \leq 2 - \tfrac{1}{76}\|a\|_4^4$, provided that $a_1 \leq \sqrt{\frac{3}{8}}$.
\end{lemma}

\begin{proof}
By the assumption, $a_k^{-2} \geq \frac{8}{3}$ for all $k$, thus, using \eqref{eq:Fourier-upbd} and \eqref{eq:Psi-quant},
\[
A_n(a) \leq 2\prod_{k=1}^n \Psi(a_k^{-2})^{a_k^2} \leq 2\prod_{k=1}^n \left(1 - \tfrac{1}{151}a_k^2\right)^{a_k^2} \leq 2\exp\Big\{-\tfrac{1}{151}\sum_{k=1}^n a_k^4\Big\}.
\]
The numerical inequality $2e^{-x} \le 2 - \frac{151}{76} x$ for $0 \le x \le \tfrac{1}{151}$ finishes the proof.
\end{proof}

\subsection{Largest weight is moderately below \texorpdfstring{$\frac{1}{\sqrt{2}}$}{1/sqrt(2)}}\label{sub:below1root2}

Suppose that $a_1 = \frac{1}{\sqrt{2}}$. Then $\Psi(a_1^{-2}) = 1$ and the Fourier-analytic bound in the proof of Lemma \ref{lm:all-weights-small} only gives that $A_n(a) \leq 2\exp\{-\frac{1}{151}\sum_{k=2}^na_k^4\}$. When $a_2$ is bounded away from $0$, this allows to conclude that $A_n(a)$ is bounded away from $2$. Otherwise, we use the Gaussian approximation for $\sum_{k=2}^n a_k\xi_k$. A toy case illustrating why this works is the vector $a = \left(\frac{1}{\sqrt2}, \frac{1}{\sqrt{2(n-1)}}, \dots, \frac{1}{\sqrt{2(n-1)}}\right)$ for large $n$. Then, if $G$ denotes a standard Gaussian random vector in $\R^4$ independent of the $\xi_j$, the central limit theorem suggests that $A_n(a)$ is well-approximated by
\[
\E\left|\frac{1}{\sqrt{2}}\xi_1 + \frac{1}{\sqrt{2}}\frac{G}{2}\right|^{-2} = 2(1 - e^{-2})
\]
(for a computation of this expectation, see \eqref{eq:BE-exp} below).
Of course, to make this heuristics quantitative, we shall use a Berry-Esseen type bound, Rai\v c's Theorem \ref{thm:Raic}. 

Thus we brake the analysis now into two further subcases.
\subsubsection{Second largest weight is small}
\begin{lemma}\label{lm:BE-case}
We have, $A_n(a) \leq 2-10^{-5}$, provided that $\sqrt{\frac{3}{8}} \leq a_1 \leq \frac{1}{\sqrt{2}}$ and $a_2 \leq 6\cdot10^{-5}$.
\end{lemma}

\begin{proof}
We let $Y = \sum_{j=2}^n a_j\xi_j$ and observe that,  by \eqref{eq:vol-via-xi} and \eqref{eq:X+Y},
\[
A_n(a) = \E\left|a_1\xi_1 + Y\right|^{-2} = \E\min\left\{a_1^{-2}, |Y|^{-2}\right\} = \int_0^{a_1^{-2}} \p{|Y|^{-2} > t} \dd t.
\]
Note that $Y$ has covariance matrix $\frac{1-a_1^2}{4}\text{Id}$. Therefore, using the Berry-Esseen bound from Theorem \ref{thm:Raic} (applied to $d=4$ and $X_j = \frac{2}{\sqrt{1-a_1^2}}a_j\xi_j$, $j = 2, \dots, n$), 
\[
\p{|Y|^{-2} > t} \leq \p{\left(\sqrt{\tfrac{1-a_1^2}{4}}|G|\right)^{-2} > t} +  (42\sqrt{2}+16)\sum_{j=2}^n \E\left|\frac{2}{\sqrt{1-a_1^2}}a_j\xi_j\right|^3,
\]
where $G$ denotes a standard Gaussian random vector in $\R^4$. Since $|G|^2$ has density $\frac{x}{4}e^{-x/2}$, $x > 0$ ($\chi^2(4)$ distribution), we obtain
\begin{align}\notag
\int_0^{a_1^{-2}} \p{\left(\sqrt{\tfrac{1-a_1^2}{4}}|G|\right)^{-2} > t} \dd t &= \E\min\left\{a_1^{-2}, \left(\sqrt{\tfrac{1-a_1^2}{4}}|G|\right)^{-2}\right\} \\\notag
&= \int_0^\infty \min\left\{a_1^{-2}, \frac{4}{1-a_1^2}\frac{1}{x}\right\} \frac{x}{4}e^{-x/2}\dd x \\\label{eq:BE-exp}
&= \frac{1}{a_1^2}\left(1 - e^{-\frac{2a_1^2}{1-a_1^2}}\right).
\end{align}
Moreover, plainly,
\[
\sum_{j=2}^n \E\left|\frac{2}{\sqrt{1-a_1^2}}a_j\xi_j\right|^3 = \frac{8}{(1-a_1^2)^{3/2}} \sum_{j=2}^n a_j^3 \leq \frac{8}{(1-a_1^2)^{3/2}}a_2\sum_{j=2}^n a_j^2 = \frac{8a_2}{\sqrt{1-a_1^2}}.
\]
Putting these together yields,
\[
A_n(a) \leq \frac{1}{a_1^2}\left(1 - e^{-\frac{2a_1^2}{1-a_1^2}}\right) + \frac{8(42\sqrt{2}+16)a_2}{a_1^2\sqrt{1-a_1^2}}.
\]
It can be checked that the first term is a decreasing function of $a_1^2$. Consequently, using $\frac{3}{8} \leq a_1^2 \leq \frac{1}{2}$ and $a_2 \leq 6\cdot 10^{-5}$, we get
\[
A_n(a) \leq \frac{8}{3}\left(1 - e^{-\frac{6}{5}}\right) + \frac{8(42\sqrt{2}+16)\cdot 6\cdot 10^{-5}}{\frac38\sqrt{\frac12}} < 2 - 10^{-5}.\qedhere
\]
\end{proof}

\subsubsection{Second largest weight is bounded away from $0$}

The goal here is to treat the case when $a_2$ is not \emph{too} small.

Note that in the following lemma instead of assuming that \eqref{eq:assumpt-a2} holds, we assume slightly less, i.e. that  $a_2 \le \frac{1-10^{-5}}{\sqrt{2}}$. We will use this in Section \ref{sub:a1_just_above}. 

\begin{lemma}\label{lm:BE-complem-case}
We have, $A_n(a) \leq 2 - 10^{-19}$, provided that $\sqrt{\frac{3}{8}} \leq a_1 \leq \frac{1}{\sqrt{2}}$ and $6\cdot 10^{-5} \leq a_2 \leq \frac{1-10^{-5}}{\sqrt{2}}$.
\end{lemma}

\begin{proof}
Note that $\Psi(a_k^{-2}) \leq 1$ for each $k$, as guaranteed by \eqref{eq:Psi-quant} since $a_k^{-2} \geq 2$ for each $k$. Using this (for all $k$ except $k=2$) in conjunction with \eqref{eq:vol-via-xi} gives
\[
A_n(a) \leq 2\prod_{k=1}^n \Psi(a_k^{-2})^{a_k^2} \leq 2\Psi(a_2^{-2})^{a_2^2}.
\]
Furthermore, again by \eqref{eq:Psi-quant},
\begin{align*}
\Psi(a_2^{-2}) \leq 1 - \min\left\{\tfrac{1}{151}a_2^2, \ \tfrac{1}{12}(a_2^{-2}-2)^2\right\} &\leq 1 - \min\left\{\tfrac{36}{151}10^{-10}, \tfrac{1}{3}((1-10^{-5})^{-2}-1)^2\right\} \\
&= 1 - \frac{36}{151}\cdot 10^{-10}.
\end{align*}
Thus,
\[
A_n(a) \leq 2\left(1 - \frac{36}{151}\cdot 10^{-10}\right)^{a_2^2} \leq 2\left(1 - \frac{36}{151}\cdot 10^{-10}a_2^2\right) < 2 - 10^{-19}.\qedhere
\]
\end{proof}

Putting Lemmas \ref{lm:BE-case} and \ref{lm:BE-complem-case} together yields the following corollary, needed in the sequel.

\begin{corollary}\label{cor:moderately-below}
We have, $A_n(a) \leq 2-10^{-19}$, provided that $\sqrt{\frac{3}{8}} \leq a_1 \leq \frac{1}{\sqrt{2}}$ and $a_2 \leq \frac{1-10^{-5}}{\sqrt{2}}$.
\end{corollary}

\subsection{Largest weight is moderately above \texorpdfstring{$\frac{1}{\sqrt{2}}$}{1/sqrt(2)}}\label{sub:a1_just_above}

\begin{lemma}\label{lm:a1-slightly-above}
We have, $A_n(a) \leq 2 - 10^{-20}$, provided that $\frac{1}{\sqrt2} < a_1 \leq \frac{1}{\sqrt2} + 6\cdot 10^{-41}$ and \eqref{eq:assumpt-a2}.
\end{lemma}

\begin{proof}
We consider the following modification of $a$, the vector
\[
b = \left(\frac{1}{\sqrt2}, \sqrt{a_1^2+a_2^2-\frac{1}{2}},a_3, \dots, a_n\right).
\]
This is a unit vector with $b_1 \geq b_2 \geq \dots \geq b_n$ and 
\[
b_2^2 \leq \left(\frac{1}{\sqrt2} + 6\cdot 10^{-41}\right)^2 + \left(\frac{1-10^{-4}}{\sqrt2}\right)^2 - \frac{1}{2} < \left(\frac{1-10^{-5}}{\sqrt{2}}\right)^2.
\]
By Lemma \ref{lm:Lip} and Corollary \ref{cor:moderately-below} applied to $b$, we get
\[
A_n(a) \leq A_n(b) + 4\sqrt{2}|a-b| \leq 2-10^{-19} + 8|a-b|.
\]
Since $\sqrt{a_1^2+a_2^2-\frac12}-a_2 = \frac{a_1^2-\frac12}{\sqrt{a_1^2+a_2^2-\frac12}+a_2} \leq \sqrt{a_1^2-\frac12}$, we have
\[
|a-b|^2 = \left(a_1 - \frac{1}{\sqrt{2}}\right)^2 + \left(\sqrt{a_1^2+a_2^2-\frac12}-a_2\right)^2 \leq 2a_1\left(a_1 - \frac{1}{\sqrt{2}}\right) < 10^{-40}
\]
and, consequently,
\[
A_ n(a) \leq 2-10^{-19} + 8\cdot 10^{-20} < 2 - 10^{-20}.
\]
\end{proof}

\subsection{Largest weight is bounded below away from \texorpdfstring{$\frac{1}{\sqrt{2}}$}{1/sqrt(2)}}\label{sub:a1_large}

\begin{lemma}\label{lm:a1-away}
We have, $A_n(a) \leq 2 - 12\sqrt{2}\cdot 10^{-41}$, provided that $a_1 \geq \frac{1}{\sqrt2} + 6\cdot 10^{-41}$.
\end{lemma}

\begin{proof}
Combining \eqref{eq:vol-via-xi} and \eqref{eq:X+Y} applied to $X = a_1\xi_1$ gives
\[
A_n(a) \leq a_1^{-2} \leq 2(1 + 6\sqrt{2}\cdot 10^{-41})^{-2} \leq 2(1-6\sqrt{2}\cdot 10^{-41}),
\]
where we used that $(1+x)^{-2} \leq 1-x$ for $x \leq \frac12$.
\end{proof}

\subsection{Putting things together}\label{sub:proof_main}

\begin{proof}[Proof of Theorem \ref{thm:main}]
Let us summarise what we proved. Without loss of generality we assume that $a$ is a unit vector such that $a_1\geq a_2\geq a_3 \geq \ldots \geq a_n \geq 0$. We considered several cases depending on the values of $a_1$ and $a_2$, which we illustrated on Figure \ref{fig:fig1} and which we discussed in Lemmas \ref{lm:loc-stab}, \ref{lm:all-weights-small}, \ref{lm:BE-case}, \ref{lm:BE-complem-case}, \ref{lm:a1-slightly-above}, \ref{lm:a1-away}. Putting them together we get that
\begin{align*}
A_n(a) &\le 2 - \min\left(\frac{1}{25}\sqrt{\delta(a)}, \frac{1}{76}\|a\|_4^4, 10^{-5}, 10^{-19}, 10^{-20}, 12\sqrt{2}\cdot 10^{-41} \right).
\end{align*}
Firstly, the first three constants are all larger than the last, so we can drop them without changing the minimum. Secondly, we use the crude bound $\delta(a) \le 2$, so that $12\sqrt{2} \cdot 10^{-41} \ge 12 \sqrt{\delta(a)} \cdot 10^{-41} \ge 10^{-40} \sqrt{\delta(a)}$. Therefore, we may rewrite this as
\begin{align*}
A_n(a) &\le 2 - \min\left(\frac{1}{25}\sqrt{\delta(a)}, \frac{1}{76}\|a\|_4^4, 12\sqrt{2}\cdot 10^{-41} \right) \\
&\le 2 - \min\left(10^{-40} \sqrt{\delta(a)} , \frac{1}{76}\|a\|_4^4 \right),
\end{align*}
which finishes the proof.
\end{proof}

{\red
\subsection*{Acknowledgements.} 
We should very much like to thank an anonymous referee for their careful reading of the manuscript and helpful suggestions, particularly the one leading to Remark \ref{rem:opt2}.
}


\begin{thebibliography}{9}


\bibitem{AS}
Abramowitz, M., Stegun, I. A., Handbook of mathematical functions with formulas, graphs, and mathematical tables. U. S. Government Printing Office, Washington, D.C., 1964.



\bibitem{Amb}
Ambrus, G., Critical central sections of the cube. 
Proc. Amer. Math. Soc. 150 (2022), no. 10, 4463--4474.

\bibitem{Ar}
Arias-de-Reyna, J.
Gaussian variables, polynomials and permanents.
Linear Algebra Appl. 285 (1998), no. 1-3, 107--114.


\bibitem{Ba}
Ball, K., Cube slicing in $R^n$. Proc. Amer. Math. Soc. 97 (1986), no. 3, 465--473.


\bibitem{Ball-BP}
Ball, K.
Some remarks on the geometry of convex sets. Geometric aspects of functional analysis (1986/87), 224--231, Lecture Notes in Math., 1317, Springer, Berlin, 1988.

\bibitem{Ball2} 
Ball, K.,
Volumes of sections of cubes and related problems, Geometric aspects of functional analysis (1987--88), 251--260, Lecture Notes in Math., 1376, Springer, Berlin, 1989. 


\bibitem{Ball-Cplank}
Ball, K.,
The complex plank problem.
Bull. London Math. Soc. 33 (2001), no. 4, 433--442.


\bibitem{BFG}
Bartha, F. \'A., Fodor, F., Gonz\'alez M. B.,
Central diagonal sections of the $n$-cube.
Int. Math. Res. Not. IMRN 2021, no. 4, 2861--2881.


\bibitem{BC-E}
Barthe, F. Cordero-Erausquin, D.,
A Gaussian correlation inequality for plurisubharmonic functions.
Preprint (2022), arXiv:2207.03847.


\bibitem{BK}
Barthe, F., Koldobsky, A.,
Extremal slabs in the cube and the Laplace transform.
Adv. Math. 174 (2003), no. 1, 89--114.


\bibitem{Bern}
Berndtsson, B.,
Pr\'ekopa's theorem and Kiselman's minimum principle for plurisubharmonic functions.
Math. Ann. 312 (1998), no. 4, 785--792.


\bibitem{Brz}
Brzezinski, P.,
Volume estimates for sections of certain convex bodies.
Math. Nachr. 286 (2013), no. 17-18, 1726--1743.


\bibitem{Bu}
Busemann, H.,
A theorem on convex bodies of the Brunn--Minkowski type.
Proc. Nat. Acad. Sci. U.S.A. 35 (1949), 27--31.



\bibitem{BP}
Busemann, H., Petty, C. M.,
Problems on convex bodies.
Math. Scand. 4 (1956), 88--94.



\bibitem{CKT}
Chasapis, G., K\"onig, H. Tkocz, T., From Ball's cube slicing inequality to Khinchin-type inequalities for negative moments. J. Funct. Anal. 281 (2021), no. 9, Paper No. 109185, 23 pp.


\bibitem{CNT}
Chasapis, G., Nayar, P., Tkocz, T., Slicing $\ell_p$-balls reloaded: Stability, planar sections in $\ell_1$. Ann. Probab. 50 (2022), no. 6, 2344--2372.



\bibitem{CST}
Chasapis, G., Singh, S., Tkocz, T.,
Haagerup's phase transition at polydisc slicing. Preprint (2022),
arXiv:2206.01026.


\bibitem{C-E1}
Cordero-Erausquin, D.,
Santal\'o's inequality on $\C^n$ by complex interpolation.
C. R. Math. Acad. Sci. Paris 334 (2002), no. 9, 767--772.


\bibitem{C-E2}
Cordero-Erausquin, D.
On Berndtsson's generalization of Pr\'ekopa's theorem.
Math. Z. 249 (2005), no. 2, 401--410.



\bibitem{DDS}
De, A., Diakonikolas, I., Servedio, R., A robust Khintchine inequality, and algorithms for computing optimal constants in Fourier analysis and high-dimensional geometry. SIAM J. Discrete Math. 30 (2016), no. 2, 1058--1094. 



\bibitem{Esk}
Eskenazis, A., On extremal sections of subspaces of $L_p$. 
Discrete Comput. Geom. 65 (2021), no. 2, 489--509.


\bibitem{ENT}
Eskenazis, A., Nayar, P., Tkocz, T., Resilience of cube slicing in $\ell_p$. Preprint (2022),  arXiv:2211.01986.


\bibitem{ENT3}
Eskenazis, A., Nayar, P., Tkocz, T., Distributional stability of the Szarek and Ball inequalities. Preprint (2023), arXiv:2301.09380.


\bibitem{Gar-tom}
Gardner, R. J., Geometric tomography. Second edition. Encyclopedia of Mathematics and its Applications, 58. Cambridge University Press, New York, 2006.


\bibitem{Haa}
Haagerup, U.,
The best constants in the Khintchine inequality.
\emph{Studia Math.} 70 (1981), no. 3, 231--283.



\bibitem{Hadw}
Hadwiger, H.,
Gitterperiodische Punktmengen und Isoperimetrie.
Monatsh. Math. 76 (1972), 410--418.




\bibitem{Hen}
Hensley, D.,
Slicing the cube in $R^n$ and probability (bounds for the measure of a central cube slice in $R^n$ by probability methods).
Proc. Amer. Math. Soc. 73 (1979), no. 1, 95--100.


\bibitem{JT}
Jenkins, J., Tkocz, T.,
Complex Hanner's Inequality for Many Functions.
Preprint (2022), arXiv:2207.09122.



\bibitem{KK}
K\"onig, H., Koldobsky, A.,
On the maximal measure of sections of the $n$-cube. Geometric analysis, mathematical relativity, and nonlinear partial differential equations, 123--155,
Contemp. Math., 599, Amer. Math. Soc., Providence, RI, 2013.


\bibitem{KoKo}
K\"onig, H., Koldobsky, A.,
On the maximal perimeter of sections of the cube. 
Adv. Math. 346 (2019), 773--804.



\bibitem{KRud}
K\"onig, H., Rudelson, M.,
On the volume of non-central sections of a cube.
Adv. Math. 360 (2020), 106929, 30 pp.



\bibitem{Kol-mon}
Koldobsky, A., Fourier analysis in convex geometry. Mathematical Surveys and Monographs, 116. American Mathematical Society, Providence, RI, 2005.



\bibitem{KKZ}
Koldobsky, A., K\"onig, H.,; Zymonopoulou, M., The complex Busemann-Petty problem on sections of convex bodies. 
Adv. Math. 218 (2008), no. 2, 352--367.


\bibitem{KPZ}
Koldobsky, A., Paouris, G., Zymonopoulou, M., Complex intersection bodies. J. Lond. Math. Soc. (2) 88 (2013), no. 2, 538--562.


\bibitem{KZ}
Koldobsky, A., Zymonopoulou, M., Extremal sections of complex $l_p$-balls, $0 < p \leq 2$. Studia Math. 159 (2003), no. 2, 185--194.


\bibitem{MR}
Melbourne, J., Roberto, C., Quantitative form of Ball's cube slicing in Rn and equality cases in the min-entropy power inequality. 
Proc. Amer. Math. Soc. 150 (2022), no. 8, 3595--3611.


\bibitem{MR-maj}
Melbourne, J., Roberto, C., Transport-majorization to analytic and geometric inequalities. 
J. Funct. Anal. 284 (2023), no. 1, Paper No. 109717.

\bibitem{MP}
Meyer, M., Pajor, A., 
Sections of the unit ball of $L^n_p$. J. Funct. Anal. 80 (1988), no. 1, 109--123.


\bibitem{NT}
Nayar, P., Tkocz, T., The unconditional case of the complex S-inequality. 
Israel J. Math. 197 (2013), no. 1, 99--106.


\bibitem{NT-surv}
Nayar, P., Tkocz, T.,
Extremal sections and projections of certain convex bodies: a survey. Preprint (2022), arXiv:2210.00885.


\bibitem{NP}
Nazarov, F., Podkorytov, A.,
Ball, Haagerup, and distribution functions. Complex analysis, operators, and related topics, 247--267, Oper. Theory Adv. Appl., 113, Birkh\"auser, Basel, 2000.



\bibitem{OP}
Oleszkiewicz, K., Pe\l czy\'nski, A.,
Polydisc slicing in $C^n$.
\emph{Studia Math.} 142 (2000), no. 3, 281--294.



{\red
\bibitem{Pou1}
Pournin, L.,
Shallow sections of the hypercube. 
Israel J. Math. 255 (2023), no. 2, 685--704.}

\bibitem{Pou2}
Pournin, L.,
Local extrema for hypercube sections. Preprint (2022), arXiv:2203.15054.


\bibitem{Ra}
Rai\v c, M.,
A multivariate Berry-Esseen theorem with explicit constants.
Bernoulli 25 (2019), no. 4A, 2824--2853.

\bibitem{Rot}
Rotem, L.,
A letter: The log-Brunn-Minkowski inequality for complex bodies.
arXiv:1412.5321.


\bibitem{Sza}
Szarek, S.,
On the best constant in the Khintchine inequality.
Stud. Math. 58, 197--208 (1976).

\bibitem{Tko}
Tkocz, T.,
Gaussian measures of dilations of convex rotationally symmetric sets in $\mathbb{C}^N$.
Electron. Commun. Probab. 16 (2011), 38--49.


\bibitem{Vaa}
Vaaler, J. D., A geometric inequality with applications to linear forms. Pacific J. Math. 83 (1979), no. 2, 543--553.



\bibitem{Zong}
Zong, C., The cube: a window to convex and discrete geometry. Cambridge Tracts in Mathematics, 168. Cambridge University Press, Cambridge, 2006.



\end{thebibliography}
\end{document}